\documentclass[11pt,leqno]{amsart}
 \usepackage{graphicx}    % needed for including graphics e.g. EPS, PS
 \usepackage{amsmath}
\usepackage{amssymb}
\usepackage[T1]{fontenc}
\usepackage[latin1]{inputenc}
\usepackage[english]{babel}
\usepackage[arrow, matrix, curve]{xy}
\usepackage{amsthm}
\usepackage{amsmath,amscd}
\usepackage{mathrsfs}
\usepackage{epic,eepic}
\numberwithin{equation}{section}

\DeclareMathOperator\GL{GL}
\DeclareMathOperator\id{id}
\DeclareMathOperator\Hom{Hom}
\DeclareMathOperator\Endo{End}
\DeclareMathOperator\rk{rk}
\DeclareMathOperator\dete{det}
\DeclareMathOperator\cf{c.f.}
\DeclareMathOperator\gr{gr}
\DeclareMathOperator\Gr{Gr}
\DeclareMathOperator\im{im}
\DeclareMathOperator\coim{coim}
\DeclareMathOperator\coker{coker}

\DeclareMathOperator\pr{pr}
\DeclareMathOperator\Fil{Fil}
\DeclareMathOperator\Isoc{Isoc}
\DeclareMathOperator\Res{Res}
\DeclareMathOperator\rig{rig}
\DeclareMathOperator\ad{ad}
\DeclareMathOperator\an{an}
\DeclareMathOperator\Rig{Rig}
\DeclareMathOperator\Ad{Ad}
\DeclareMathOperator\lft{lft}

\renewcommand{\phi}{\varphi}

\newcommand{\Z}{\mathbb{Z}}
\newcommand{\Q}{\mathbb{Q}}
\newcommand{\R}{\mathbb{R}}

\newcommand{\Abb}{\mathbb{A}}

\newcommand{\Gbb}{\mathbb{G}}
\newcommand{\Pbb}{\mathbb{P}}

\newcommand{\Dcal}{\mathcal{D}}
\newcommand{\Ecal}{\mathcal{E}}
\newcommand{\Fcal}{\mathcal{F}}
\newcommand{\Gcal}{\mathcal{G}}
\newcommand{\Ocal}{\mathcal{O}}

\newcommand{\Dfrak}{\mathfrak{D}}
\newcommand{\afrak}{\mathfrak{a}}

\newcommand{\rfH}{\mathscr{H}}

\newtheorem{theo}{Theorem}[section]
\newtheorem{lem}[theo]{Lemma}
\newtheorem{prop}[theo]{Proposition}
\newtheorem{cor}[theo]{Corollary}
\theoremstyle{remark}
\newtheorem{rem}[theo]{Remark}
\theoremstyle{remark}
\newtheorem{expl}[theo]{Example}
\theoremstyle{definition}
\newtheorem{defn}[theo]{Definition}

\begin{document}
\title[Weakly admissible filtered $\phi$-modules and the adjoint quotient]{On families of weakly admissible filtered $\phi$-modules and the adjoint quotient of $\GL_d$}
\author[E. Hellmann]{Eugen Hellmann}
\begin{abstract}
We study the relation of the notion of weak admissibility in families of filtered $\phi$-modules, as considered in \cite{Hellmann}, with the adjoint quotient.
We show that the weakly admissible subset is an open subvariety in the fibers over the adjoint quotient. Further we determine the image of the weakly admissible set in the adjoint quotient generalizing earlier work of Breuil and Schneider. 
\end{abstract}
\maketitle

\section{Introduction}

Filtered $\phi$-modules appear in $p$-adic Hodge-theory as a category of linear algebra data describing \emph{crystalline} representations of the absolute Galois group of a local $p$-adic field. More precisely, there is an equivalence of categories between crystalline representations and \emph{weakly admissible} filtered $\phi$-modules, see \cite{ColmezFont}. Here weak admissibility is a semi-stability condition relating the slopes of the $\phi$-linear endomorphism $\Phi$ with the filtration. 

In our companion paper \cite{Hellmann} we define and study arithmetic families of filtered $\phi$-modules and crystalline representations. Our families are parameterized by rigid analytic spaces or adic spaces in the sense of Huber. We show that the condition of being weakly admissible is an open condition \cite[Theorem 1.1]{Hellmann} and that there is an open subset of the weakly admissible locus over which there exists a family of crystalline representations giving rise to the family of filtered $\phi$-modules \cite[Theorem 1.3]{Hellmann}.

In this paper we study the weakly admissible locus in more detail. In the setting of period domains in the sense of Rapoport and Zink \cite{RapoZink}, the weakly admissible locus is an admissible open subset of a flag variety.  Contrarily, the weakly admissible locus in our set up has an algebraic nature as soon as we fix the Frobenius $\Phi$, or even the conjugacy class of its semi-simplification.
Further we analyze the image of the weakly admissible locus in the adjoint quotient. The question whether there exists a weakly admissible filtration for a fixed conjugacy class of the semisimplification of the Frobenius already appears in work of Breuil and Schneider \cite{BreuilSchneider} on the $p$-adic Langlands correspondence. Unlike the characterization in \cite{BreuilSchneider}, our characterization of the set of automorphisms $\Phi$ for which there exists a weakly admissible filtration is purely in terms of the adjoint quotient of $\GL_d$.

Our main results are as follows:
Fix a finite extension $K$ of $\Q_p$ and write $K_0$ for the maximal unramified extension of $\Q_p$ inside $K$. Let $d>0$ be an integer and denote by $A\subset \GL_d$ the diagonal torus. For a dominant cocharacter 
\[\nu: \Gbb_{m,\bar\Q_p}\longrightarrow (\Res_{K/\Q_p}A)_{\bar\Q_p}\]
we write $\Gr_\nu$ for the partial flag variety of $\Res_{K/\Q_p}\GL_d$ parametrizing flags of "type $\nu$". This variety is defined over the reflex field $E$ of $\nu$. As in \cite[4.1]{Hellmann} we denote by 
\[\Dfrak_\nu= \big((\Res_{K_0/\Q_p}\GL_d)_E\times\Gr_\nu\big)/(\Res_{K_0/\Q_p}\GL_d)_E\]
the stack of filtered $\phi$-modules with filtration of "type $\nu$" on the category of adic spaces locally of finite type. Let $W$ denote the Weyl group of $\GL_d$. We will define a morphism
\[\alpha:\Dfrak_\nu\longrightarrow (A/W)^{\ad}\]
to the adification of the adjoint quotient $A/W$ and prove the following theorem.
\begin{theo}\label{maintheo1}
Let $x\in (A/W)^{\rm ad}$ and form the $2$-fiber product
\[\begin{xy}\xymatrix{
\alpha^{-1}(x)^{\rm wa} \ar[r]\ar[d] & \Dfrak_\nu^{\rm wa}\ar[d]\\
x \ar[r] & (A/W)^{\rm ad}.
}
\end{xy}\]
Then there exists a finite extension $F$ of $\Q_p$ inside $k(x)$ and an Artin stack in schemes $\mathfrak{A}$ over $F$ such that 
\[\alpha^{-1}(x)^{\rm wa}=\mathfrak{A}^{\rm ad}\otimes_F k(x).\]
The stack $\mathfrak{A}$ is the stack quotient of a quasi-projective $F$-variety.
\end{theo}
Further we determine the image of the weakly admissible locus $\Dfrak_\nu^{\rm wa}$ under the morphism $\alpha$.
\begin{theo}\label{maintheo2}
Let $\nu$ be a dominant coweight as above. There is a dominant coweight $\mu(\nu)$ of $\GL_d$ associated to $\nu$ such that 
\[\alpha^{-1}(x)^{\rm wa}\neq\emptyset\Longleftrightarrow x\in(A/W)_{\leq\mu(\nu)}.\]
Here $(A/W)_{\leq\mu(\nu)}$ is a Newton-stratum in the sense of Kottwitz \cite{Kottwitz}.
\end{theo}
The coweight $\mu(\nu)$ which appears in the theorem is explicit and defined in Definition $\ref{defmunu}$

{\bf Acknowledgements}: I thank my advisor M. Rapoport for his interest and advice.
This work was supported by the SFB/TR 45 "Periods, Moduli Spaces and Arithmetic of Algebraic Varieties" of the DFG (German Research Foundation).

\section{Filtered $\phi$-modules}

Throughout this section we denote by $F$ a topological field containing $\Q_p$ with a continuous valuation $v_F:F\rightarrow \Gamma_F\cup\{0\}$ in the sense of \cite[2, Definition]{contval} that is $\Gamma_F$ is a totally ordered abelian group (written multiplicative) and
\begin{align*}
 v(0)&=0 \\
 v(1)&=1 \\
 v(ab)&=v(a)v(b)\\
 v(a+b)&\leq\max\{v(a),v(b)\},
\end{align*}
where the order on $\Gamma_F$ is extended to $\Gamma_F\cup\{0\}$ by $0<\gamma$ for all $\gamma\in\Gamma_F$.
We will introduce the notion of a filtered $\phi$-module with coefficients in $F$ and define weak admissibility for these objects.\\
Recall that $K_0$ is an unramified extension of $\Q_p$ with residue field $k$ and write $f=[K_0:\Q_p]$. We write $\phi$ for the lift of the absolute Frobenius to $K_0$. 
\subsection{$\phi$-modules with coefficients}
In this subsection define and study what we call \emph{isocrystals over $k$ with coefficients in $F$}.
\begin{defn}\label{defisoc}
An \emph{isocrystal over $k$ with coefficients in $F$} is a free $F\otimes_{\Q_p}K_0$-module $D$ of finite rank together with an automorphism $\Phi:D\rightarrow D$ that is semi-linear with respect to $\id\otimes\phi:F\otimes_{\Q_p}K_0\rightarrow F\otimes_{\Q_p}K_0$.\\
A morphism $f:(D,\Phi)\rightarrow (D',\Phi')$ is an $F\otimes_{\Q_p}K_0$-linear map $f:D\rightarrow D'$ such that
\[f\circ \Phi= \Phi'\circ f.\]
The category of isocrystals over $k$ with coefficients in $F$ is denoted by $\Isoc(k)_F$.
\end{defn}
It is easy to see that $\Isoc(k)_F$ is an $F$-linear abelian $\otimes$-category with the obvious notions of direct sums and tensor products.
\begin{rem}
\noindent (i) Given an $F\otimes_{\Q_p}K_0$-module $D$ of finite type, the existence of a semi-linear automorphism $\Phi:D\rightarrow D$ implies that $D$ is free over $F\otimes_{\Q_p}K_0$. This fact will be used in the sequel.\\
\noindent (ii) In the classical setting an isocrystal over $k$ is a finite-dimensional $K_0$-vector space with $\phi$-linear automorphism $\Phi$, i.e. an object in $\Isoc(k)_{\Q_p}$.

If $F$ is finite over $\Q_p$, then an isocrystal over $k$ with coefficients in $F$ is the same as an object $(D,\Phi)\in\Isoc(k)_{\Q_p}$ together with a map
\[
 F\longrightarrow \Endo_\Phi(D),
\]
where the subscript $\Phi$ on the right hand side indicates that the endomorphisms commute with $\Phi$ (compare \cite[VIII, 5]{DatOrlikRapo} for example).
This is clearly equivalent to our definition.
\end{rem}
Let $F'$ be an extension of $F$ with valuation $v_{F'}:F'\rightarrow \Gamma_{F'}\cup\{0\}$ extending the valuation $v_F$. The \emph{extension of scalars} from $F$ to $F'$ is the functor
\begin{equation}\label{extension1}
-\otimes_FF':\Isoc(k)_F\longrightarrow \Isoc(k)_{F'}
\end{equation} 
that maps $(D,\Phi)\in\Isoc(k)_F$ to the object $(D\otimes_{F}F',\Phi\otimes\id)$.

If $F'$ is a finite extension of $F$, then we also define the \emph{restriction of scalars}
\begin{equation}\label{restriction1}
\epsilon_{F'/F}:\Isoc(k)_{F'}\longrightarrow \Isoc(k)_F.
\end{equation}
This functor maps $(D',\Phi')\in\Isoc(k)_{F'}$ to itself, forgetting the $F'$-action but keeping the $F$-action.

We write $\Gamma_F\otimes\Q$ for the localisation of the abelian group $\Gamma_F$. Then every element $\gamma'\in\Gamma_F\otimes\Q$ can be written as a single tensor $\gamma\otimes r$ and we extend the total order of $\Gamma_F$ to $\Gamma_F\otimes\Q$ by 
\[a\otimes\tfrac{1}{m}<b\otimes\tfrac{1}{n}\Leftrightarrow a^n<b^m.\]

\begin{defn}
Let $(D,\Phi)\in\Isoc(k)_F$ and $d=\rk_{F\otimes_{\Q_p}K_0}D$. 
The map $\Phi^f:D\rightarrow D$ is an $F$-linear automorphism of the $fd$-dimensional $F$-vector space $D$.\\
(i) Define the \emph{Newton slope} of $(D,\Phi)$ as
\[
 \lambda_N^{(F)}(D,\Phi)=v_F(\dete_F \Phi^f)\otimes \tfrac{1}{f^2d}\in\Gamma_F\otimes\Q.
\]
Here $\dete_F$ means that we take the determinant of an $F$-linear map on an $F$-vector space.\\
(ii) Let $\lambda\in\Gamma_F\otimes\Q$. An object $(D,\Phi)\in\Isoc(k)_F$ is called \emph{purely of Newton-slope $\lambda$} if for all $\Phi$-stable $F\otimes_{\Q_p}K_0$-submodules $D'\subset D$ we have $\lambda_N^{(F)}(D',\Phi|_{D'})=\lambda$. 
\end{defn}
\begin{lem}\label{extensionandslop1}
Let $F'$ be an extension of $F$ with valuation $v_{F'}$ extending $v_F$ and $(D,\Phi)\in\Isoc(k)_F$. Then
\[\lambda_N^{(F')}(D\otimes_FF',\Phi\otimes\id)= \lambda_N^{(F)}(D,\Phi).\]
If in addition $F'$ is finite over $F$ and $(D',\Phi')\in\Isoc(k)_{F'}$, then
\[\lambda_N^{(F)}(\epsilon_{F'/F}(D',\Phi'))=\lambda_N^{(F')}(D',\Phi').\]  
\end{lem}
\begin{proof}
These are straightforward computations.
\end{proof}
As the Newton slope is preserved under extension and restriction of scalars we will just write $\lambda_N$ in the sequel.
\begin{rem}\label{classicalslope}
Let $(D,\Phi)\in\Isoc(k)_{\Q_p}$ be an object of rank $d$ and denote for the moment by $v_p$ the usual $p$-adic valuation on $\Q_p$. Write $|-|=p^{-v_p(-)}$.
Then the value group of the absolute value is $\Gamma_{\Q_p}=p^\Z$ and we identify $\Gamma_{\Q_p}\otimes\Q$ with the subgroup $p^\Q$ of $\R\backslash\{0\}$.
Our definitions then imply
\begin{equation}\label{slopeclassical}
\lambda_N(D,\Phi)=p^{-\tfrac{1}{d} v_p(\dete_{K_0}\Phi)}.
\end{equation}
Here $v_p(\dete_{K_0}\Phi)$ is the $p$-adic valuation of the determinant over $K_0$ of any matrix representing the semi-linear map $\Phi$ in some chosen basis.
This matrix is well defined up to $\phi$-conjugation and hence the valuation of the determinant is independent of choices.
Note that (the negative of) the exponent in $(\ref{slopeclassical})$ is the usual Newton slope of the isocrystal $(D,\Phi)$ over $k$, compare \cite{Zink} for example.
\end{rem}
\begin{prop}\label{slopedecomp}
Let $(D,\Phi)\in\Isoc(k)_F$, then there exist unique elements $\lambda_1<\lambda_2<\dots<\lambda_r\in\Gamma_F\otimes\Q$ and a unique decomposition
\[D=D_1\oplus D_2\oplus\dots\oplus D_r\]
of $D$ into $\Phi$-stable $F\otimes_{\Q_p}K_0$-submodules such that $(D_i,\Phi|_{D_i})$ is purely of Newton slope $\lambda_i$.
\end{prop}
\begin{proof}
First we show the existence of such a decomposition. The uniqueness will then follow from Lemma $\ref{lemabbNewton}$ below.\\
\emph{Step 1:} Assume first that there exists an embedding $\psi_0:K_0\hookrightarrow F$. \\
We obtain an isomorphism
\[\begin{xy}
   \xymatrix{
F\otimes_{\Q_p}K_0 \ar[r]^{\hspace{-0.2cm}\cong} & \prod_{\psi:K_0\rightarrow F}F.
}
  \end{xy}
\]
The endomorphism $\id\otimes\phi$ on the left hand side translates to the shift of the factors on the right hand side. Further we obtain the corresponding decomposition 
\[
 D=\prod_{\psi}V_\psi
\]
into $F$-vector spaces $V_\psi$ and $F$-linear isomorphisms 
\[\begin{xy}\xymatrix{
 \Phi_\psi=\Phi|_{V_\psi}:V_\psi\ar[r]^{\hspace{0.5cm}\cong} & V_{\psi\circ\phi}.
}\end{xy}
\]
There is a bijection between the $\Phi$-stable subspaces $D'$ of $D$ and the $\Phi^f|_{V_{\psi_0}}$-stable subspaces of $V_{\psi_0}$ given by $D'\mapsto D'\cap V_{\psi_0}$. \\
Given $D'\subset D$ and $U=D'\cap V_{\psi_0}$ we have
\[\lambda_N(D',\Phi|_{D'})=v_F(\dete_F\Phi^f|_U)\otimes\tfrac{1}{f\dim_FU}\in\Gamma_F\otimes\Q.\]
Hence the desired decomposition of $D$ is induced by the decomposition of $V_{\psi_0}$ into the maximal $\Phi^f|_{V_{\psi_0}}$-stable subspaces $U\subset V_{\psi_0}$ such that
\[v_F(\dete_F\Phi^f|_{U'})\otimes\tfrac{1}{f\dim_FU'}=\lambda_i\]
for all $\Phi^f$-stable subspaces $U'\subset U$.\\
\emph{Step 2:} If there is no embedding $\psi$ of $K_0$ into $F$, then we find a finite extension $F'=FK_0$ of $F$ such that $K_0$ embeds into $F'$. We want to deduce the result from Step 1 by Galois descent. We replace $F'$ by its Galois hull and extend the valuation from $F$ to $F'$ by setting $v_F(\Ocal_{K_0}^\times)=\{1\}$, where $\Ocal_{K_0}\subset K_0 $ is the ring of integers.\\
Write $(D',\Phi')$ for the extension of scalars of $(D,\Phi)\in\Isoc(k)_F$ to $\Isoc(k)_{F'}$. Then there exists $\lambda_1>\lambda_2>\dots>\lambda_r\in\Gamma_F\otimes\Q$ and a decomposition 
\begin{equation}\label{Dstrichdecomp}
 D'=D'_1\oplus D'_2\oplus\dots D'_r
\end{equation}
such that the $D'_i$ are $\Phi'$-stable and $(D'_i,\Phi'|_{D'_i})$ is purely of slope $\lambda_i$.
Now the action of the Galois group ${\rm Gal}(F'/F)$ preserves the valuation on $F'$ and hence also the Newton slope of a $\Phi'$-stable subobject of $D'$. It follows that ${\rm Gal}(F'/F)$ preserves the decomposition $(\ref{Dstrichdecomp})$ and hence this decomposition descends to $D$.
\end{proof}
\begin{rem}
Proposition $\ref{slopedecomp}$ replaces the slope decomposition in the classical context ($\cf.$ \cite[VI, 3]{Zink} for example).
\end{rem}
\begin{defn}
 Let $(D,\Phi)\in \Isoc(k)_F$ and denote by $D=\bigoplus D_i$ a decomposition of $D$ into $\Phi$-stable submodules purely of slope $\lambda_i\in\Gamma_F\otimes\Q$ as in Proposition $\ref{slopedecomp}$. We will refer to this as the \emph{slope decomposition}. Further, for $\lambda\in\Gamma_F\otimes\Q$ we define
\[
D_\lambda=
 \begin{cases}
  D_{i} &,\ \lambda=\lambda_i\\
  0 & \text{otherwise}.
 \end{cases}
\]
\end{defn}
\begin{lem}\label{lemabbNewton}
 Let $f:(D,\Phi)\rightarrow (D',\Phi')$ be a morphism in $\Isoc(k)_F$. Consider slope decompositions $D=\bigoplus D_i$ and $D'=\bigoplus D'_j$ as in Proposition $\ref{slopedecomp}$. Then for all $\lambda\in\Gamma_F\otimes\Q$
\[
 f(D_\lambda)\subset D'_\lambda.
\]
\end{lem}
\begin{proof}
 This is an immediate consequence of $f\circ \Phi=\Phi'\circ f$.
\end{proof}

\subsection{Filtered Isocrystals with coefficients}
Recall the $K$ is a totally ramified extension of $K_0$. We denote by $e=[K:K_0]$ the ramification index of $K$.
In this section we define the basic object of our study.
\begin{defn}\label{deffilisoc}
 A \emph{$K$-filtered isocrystal over $k$ with coefficients in $F$} is a triple $(D,\Phi,\Fcal^\bullet)$, where $(D,\Phi)\in\Isoc(k)_F$ and $\Fcal^\bullet$ is a descending, separated and exhaustive $\Z$-filtration on $D_K=D\otimes_{K_0}K$ by (not necessarily free) $F\otimes_{\Q_p}K$-submodules. \\
A morphism 
\[f:(D,\Phi,\Fcal^\bullet)\longrightarrow (D',\Phi',\Fcal'^\bullet)\] 
is a morphism $f:(D,\Phi)\rightarrow (D',\Phi')$ in $\Isoc(k)_F$ such that $f\otimes\id:D_K\rightarrow D'_K$ respects the filtrations.\\
The category of $K$-filtered isocrystals over $k$ with coefficients in $F$ is denoted by $\Fil\Isoc(k)_F^K$.
\end{defn}
It is easy to see that $\Fil\Isoc(k)_F^K$ is an $F$-linear $\otimes$-category. Further it has obvious notions of kernels, cokernels and exact sequences. For an extension $F'$ of $F$ we again have an \emph{extension of scalars} like in $(\ref{extension1})$,
\[-\otimes_FF':\Fil\Isoc(k)_F^K\longrightarrow \Fil\Isoc(k)_{F'}^K.\]
If $F'$ is finite over $F$, we also have a \emph{restriction of scalars} like in $(\ref{restriction1})$,
\[\epsilon_{F'/F}:\Fil\Isoc(k)_{F'}^K\longrightarrow \Fil\Isoc(k)_F^K.\]
In the following we will often shorten our notation and just write $D$ for an object $(D,\Phi,\Fcal^\bullet)\in\Fil\Isoc(k)_F^K$.

We now want to develop a slope theory for filtered isocrystals and define weakly admissible objects. 
\begin{defn}
Let $(D,\Phi,\Fcal^\bullet)\in\Fil\Isoc(k)_F^K$. We define  
\begin{align*}
\deg\Fcal^\bullet&=\sum_{i\in\Z}\tfrac{1}{ef}i\dim_F\gr_i\Fcal^\bullet\\
\deg_F(D)&=(v_F(\dete_F\Phi^f)\otimes\tfrac{1}{f^2})^{-1}\ v_F(p)^{\deg(\Fcal^\bullet)}\in\Gamma_F\otimes\Q\\
\mu_F(D)&=\deg_F(D)(1\otimes\tfrac{1}{d})\in\Gamma_F\otimes\Q.
\end{align*}
We call $\mu_F(D)$ the \emph{slope} of $D$.
\end{defn}
\begin{rem}
As in Lemma $\ref{extensionandslop1}$, one easily sees that the slope $\mu_F$ is preserved under extension and restriction of scalars. Hence we will just write $\mu$ in the sequel.
\end{rem}
Now we have a Harder-Narasimhan formalism as in \cite[Chapter 1]{DatOrlikRapo}. The only difference is that our valuations are written multiplicatively, while in the usual theory they are written additively. We will only sketch the proofs and refer to \cite{DatOrlikRapo} for the details.
\begin{lem}\label{lemseseq}
Let
\[
\begin{xy}
\xymatrix{
0 \ar[r] & D'\ar[r] &D\ar[r] &D''\ar[r] &0
}
\end{xy}
\]
be a short exact sequence in $\Fil\Isoc(k)_F^K$. Then
\[\deg_F(D)=\deg_F(D')\deg_F(D'').\]
Further
\[
\max\{\mu(D'),\mu(D'')\}\geq \mu(D)\geq \min\{\mu(D'),\mu(D'')\}.
\]
The sequence $\mu(D'),\mu(D),\mu(D'')$ is either strictly increasing or strictly decreasing or stationary.
\end{lem}
\begin{proof}
The first assertion is obvious from the definitions and the second is a direct consequence.
\end{proof}
\begin{lem}\label{lemcoimim}
Let $f:D\rightarrow D'$ be a morphism in $\Fil\Isoc(k)_F^K$. Then
\[\deg_F({\rm coim}\,f)\geq \deg_F({\rm im}\,f).\]
\end{lem}
\begin{proof}
Replacing $D$ by ${\rm coim}\,f$ and $D'$ by ${\rm im}\,f$, we may assume that $f$ is an isomorphism in $\Isoc(k)_F$. Now the assertion follows easily from 
\[(f\otimes \id) (\Fcal^i)\subset \Fcal'^i.\]
\end{proof}
\begin{defn}
An object $(D,\Phi,\Fcal^\bullet)\in\Fil\Isoc(k)_F^K$ is called \emph{semi-stable} if, for all $\Phi$-stable subobjects $D'\subset D$, we have $\mu(D')\geq \mu(D)$.
It is called \emph{stable} if the inequality is strict for all proper subobjects. Finally $D$ is called \emph{weakly admissible} if it is semi-stable of slope $1$.
\end{defn}
Note that semi-stability is defined using "$\geq$" instead of "$\leq$" (as in \cite{DatOrlikRapo}), since our valuations are written multiplicatively.
\begin{rem}
Let $(D,\Phi,\Fcal^\bullet)\in\Fil\Isoc(k)_{\Q_p}^K$. Using the notations of Remark $\ref{extensionandslop1}$, we find
\[\mu(D)=p^{v_p(\dete_{K_0}\Phi)-\sum_ii\dim_K(\Fcal^i/\Fcal^{i+1})}.\]
Hence we see that $D$ is weakly admissible if and only it is weakly admissible in the sense of \cite[3.4]{ColmezFont}. 
\end{rem}
\begin{prop}
Let $D,D'\in\Fil\Isoc(k)_F^K$ be semi-stable objects.\\
\noindent {\rm (i)} If $\mu(D)<\mu(D')$, then $\Hom(D,D')=0$.\\
\noindent {\rm (ii)} If $\mu(D)=\mu(D')=\mu$, then for all $f\in\Hom(D,D')$ we have $\im\,f\cong\coim\,f$ and the objects
$\ker f$, $\coker f$ and $\im f$ are semi-stable of slope $\mu$.
\end{prop}
\begin{proof}
The proof is the same as in \cite[Proposition 1.1.20]{DatOrlikRapo}
\end{proof}
\begin{cor}
Let $\mu\in\Gamma\otimes\Q$, then the full subcategory of $\Fil\Isoc(k)_F^K$ consisting of semi-stable objects of slope $\mu$ is an abelian, artinian and noetherian category which is stable under extensions. The simple objects are exactly the stable ones.
\end{cor}
\begin{proof}
The proof is the same as the proof of \cite[Corollary 1.2.21]{DatOrlikRapo}.
\end{proof}

The main result of this section is the existence of a Harder-Narasimhan filtration for the objects in $\Fil\Isoc(k)_F^K$. The existence of this filtration will also imply that semi-stability (and hence weak admissibility) is preserved under extension and restriction of scalars.

\begin{prop}\label{HNfilt}
Let $D\in\Fil\Isoc(k)_F^K$, then there exist unique elements $\mu_1<\mu_2<\dots<\mu_r\in\Gamma_F\otimes\Q$ and a unique filtration
\[0=D_0\subset D_1\subset D_2\subset\dots\subset D_r=D\]
of $D$ in $\Fil\Isoc(k)_F^K$ such that $D_i/D_{i-1}$ is semi-stable of slope $\mu_i$.
\end{prop}
\begin{proof}
The proof is similar to the proof of \cite[Proposition 1.3.1 (a)]{DatOrlikRapo}. \\
First we prove the existence of the filtration. The uniqueness will then follow from Lemma $\ref{lemabbHN}$ below. \\
By the existence of the slope decomposition in Proposition $\ref{slopedecomp}$, the set
\[\{\mu(D')\mid D'\subset D\ \text{stable under}\ \Phi\}\]
is finite. Hence there is a unique minimal element $\mu_1$ and we claim that there is a maximal subobject $D_1\subset D$ of slope $\mu_1$ which then must be semi-stable.
This follows, as the sum of two subobjects of slope $\mu_1$ has again slope $\mu_1$, by Lemma $\ref{lemseseq}$ and the minimality of $\mu_1$.\\
Proceeding with $D/D_1$ the claim follows by induction.
\end{proof}
\begin{defn}\label{defHN}
Let $D\in\Fil\Isoc(k)_F^K$ and denote by 
\[0=D_0\subset D_1\subset D_2\subset\dots\subset D_r=D\]
a filtration as in Proposition $\ref{HNfilt}$. This filtration is called the \emph{Harder-Narasimhan filtration} of $D$. For $\mu\in\Gamma_F\otimes\Q$ we define
\[
D_{(\mu)}=
\begin{cases} 0& \text{if}\ \mu<\mu_1 \\
 D_i &\text{if}\ \mu_i\leq \mu<\mu_{i+1}\\
 D &\text{if}\ \mu\geq \mu_r.
 \end{cases}
\] 
\end{defn}
\begin{lem}\label{lemabbHN}
Let $f:D\rightarrow D'$ be a morphism in $\Fil\Isoc(k)_F^K$ and fix filtrations of $D$ and $D'$ as in Proposition $\ref{HNfilt}$. Let $\mu\in\Gamma_F\otimes\Q$, then {\rm (}with the notation of Definition $\ref{defHN}${\rm )}:
\[f(D_{(\mu)})\subset D'_{(\mu)}.\]
\end{lem}
\begin{proof}
The proof is the same as in \cite[Proposition 1.3.1 (b)]{DatOrlikRapo}.
\end{proof}
\begin{cor}\label{waextstab}
Let $F'$ be an extension of $F$ with valuation $v_{F'}$ extending $v_F$ and $D\in\Fil\Isoc(k)_F^K$. 
If $D$ is semi-stable of slope $\mu$, then $D'=D\otimes_FF'$ is semi-stable of slope $\mu$.\\
If in addition $F'$ is finite over $F$ and $D'\in\Fil\Isoc(k)_{F'}^K$ is semi-stable of slope $\mu$, then $\epsilon_{F'/F}(D')\in\Fil\Isoc(k)_F^K$ is semi-stable of slope $\mu$. 
\end{cor}
\begin{proof}
We may assume that $F'$ is finitely generated over $F$ as every counterexample for the semi-stability condition is defined over a finitely generated extension. Then $F'$ is an algebraic extension of a purely transcendental extension and we can treat both cases separately. 

Assume first that $F'$ is an algebraic extension of $F$. We may replace it by its Galois hull and denote by $G={\rm Gal}(F'/F)$ the Galois group of $F'$ over $F$. Then the action of $G$ preserves the valuation on $F'$. We denote by 
\[0=D'_0\subset D'_1\subset D'_2\subset\dots\subset D'_r=D'\]
the Harder-Narasimhan filtration of $D'$. The action of $G$ commutes with $\Phi$ and preserves the filtration $\Fcal^\bullet\otimes_FF'$ of $D'\otimes_{K_0}K$. It follows that it preserves the slope of a $\Phi$-stable subobject and hence preserves the Harder-Narasimhan filtration. It follows that the filtration descends to $F$ and hence it can only have one step, as $D$ is semi-stable.

Assume now that $F'$ is purely transcendental over $F$.  Again we write $G={\rm Aut}(F'/F)$ for the group of $F$-automorphisms of $F'$. As above we only need to check that $G$ preserves the slope of a $\Phi$-stable subobject of $D'$. Let $U\subset D'$ be such a $\Phi$-stable subspace. Then $U$ is a direct sum of indecomposable $\Phi$-modules $U_i$ such that the isomorphism class of $U_i$ is defined over $F$, as $F$ is algebraically closed in $F'$. It follows that $\dete_{F'}(\Phi^f|_U)\in F$ and hence the action of $G$ preserves the slope of $U$.

Now assume that $F'$ is finite over $F$ and $D'$ is a semi-stable object of $\Fil\Isoc(k)_{F'}^K$. Consider the Harder-Narasimhan filtration of $\epsilon_{F'/F}(D')$. By Lemma $\ref{lemabbHN}$ the filtration steps are stable under the operation of $F'$. Hence the filtration can have only one step.
\end{proof}

\section{Families of filtered $\phi$-modules}

It is shown in \cite[4]{Hellmann} that the stack of weakly admissible filtered $\phi$-modules is an open substack of the stack of filtered $\phi$-modules. We briefly recall this result before we study the weakly admissible locus in the fibers over the adjoint quotient. We write $\Rig_E$ for the category of rigid analytic spaces over a finite extension $E$ of $\Q_p$ (see \cite{BoschGR}) and $\Ad^{\lft}_E$ for the category of adic spaces locally of finite type over $E$, see \cite{Hu2}.

\subsection{Stacks of filtered $\phi$-modules}
Let $d$ be a positive integer and $\nu$ an algebraic cocharacter 
\begin{equation}\label{cocharnu}
\nu:\bar\Q_p^\times\longrightarrow (\Res_{K/\Q_p}A_K)(\bar\Q_p),
\end{equation}
where $A\subset \GL_d$ is the diagonal torus. We assume that this cocharacter is dominant with respect to the restriction $B$ of the Borel subgroup of upper triangular matrices in $(\GL_d)_K$. 
This cocharacter is defined over the reflex field $E\subset \bar\Q_p$. 
Let $\Delta$ denote the set of simple roots (defined over $\bar\Q_p$) of $\Res_{K/\Q_p}\GL_d$ with respect to $B$ and denote by $\Delta_\nu\subset \Delta$ the set of all simple roots $\alpha$ such that $\langle \alpha,\nu\rangle=0$. Here $\langle -,-\rangle$ is the canonical pairing between characters and cocharacters.
We write $P_\nu$ for the parabolic subgroup of $(\Res_{K/\Q_p}\GL_d)$ containing $B$ and corresponding to $\Delta_\nu\subset \Delta$. This parabolic subgroup is defined over $E$, and the quotient by this parabolic is a projective variety over $E$,
\begin{equation}
\Gr_{K,\nu}=(\Res_{K/\Q_p}\GL_d)_E/P_\nu
\end{equation}
representing the functor 
\[S\mapsto\{\text{filtrations}\ \Fcal^\bullet\ \text{of}\ \Ocal_S\otimes_{\Q_p}K^d\ \text{of type}\ \nu\}\]
on the category of $E$-schemes.  Here the filtrations are locally on $S$ direct summands. Being of \emph{type} $\nu$ means the following. Assume that the cocharacter 
\[\nu:\bar\Q_p^\times\longrightarrow \prod_{\psi:K\rightarrow \bar\Q_p}\GL_d(\bar\Q_p)\]
is given by cocharacters
\[\nu_\psi: \lambda\mapsto {\rm diag}((\lambda^{i_1(\psi)})^{(m_1(\psi))},\dots,(\lambda^{i_r(\psi)})^{(m_r(\psi))})\]
for some integers $i_j(\psi)\in\Z$ and multiplicities $m_j(\psi)>0$. Then any point $\Fcal^\bullet\in\Gr_{K,\nu}(\bar\Q_p)$ is a filtration $\prod_\psi\Fcal_\psi^\bullet$ of $\prod_\psi\bar\Q_p^d$ such that
\[\dim_{\bar\Q_p}\gr_i(\Fcal_\psi^\bullet)=\begin{cases}
0& \text{if}\ i\notin\{i_1(\psi),\dots,i_r(\psi))\}\\
m_j(\psi) & \text{if}\ i=i_j(\psi). 
\end{cases}\] 
We denote by $\Gr_{K,\nu}^{\rig}$ resp. $\Gr_{K,\nu}^{\ad}$ the associated rigid space, resp. the associated adic space (cf. \cite[9.3.4]{BoschGR}).

Given $\nu$ as in $(\ref{cocharnu})$ and denoting as before by $E$ the reflex field of $\nu$, we consider the following fpqc-stack $\Dfrak_\nu$ on the category ${\Rig}_E$ (resp. on the category ${\Ad}_E^{\lft}$).
For $X\in{\Rig}_E$ (resp. ${\Ad}_E^{\lft}$) the groupoid $\Dfrak_\nu(X)$ consists of triples $(D,\Phi,\Fcal^\bullet)$, where $D$ is a coherent $\Ocal_X\otimes_{\Q_p}K_0$-modules
which is locally on $X$ free over $\Ocal_X\otimes_{\Q_p}K_0$ and $\Phi:D\rightarrow D$ is an $\id\otimes\phi$-linear automorphism. Finally $\Fcal^\bullet$ is a filtration of $D_K=D\otimes_{\Q_p}K$ of type $\nu$, i.e. after choosing fpqc-locally on $X$ a basis of $D$, the filtration $\Fcal^\bullet$ induces a map to $\Gr_{K,\nu}^{\rig}$ (resp. $\Gr_{K,\nu}^{\ad}$), compare also \cite[5.a]{phimod}.\\ 

One easily sees that the stack $\Dfrak_\nu$ is the stack quotient of the rigid space
\begin{equation}\label{Xnu}
X_\nu=(\Res_{K_0/\Q_p}\GL_d)_E^{\rig}\times\Gr_{K,\nu}^{\rig}
\end{equation}
by the $\phi$-conjugation action of $(\Res_{K_0/\Q_p}\GL_d)_E^{\rig}$ given by 
\begin{equation}\label{phiconjaction}
(A,\Fcal^\bullet)\cdot g=(g^{-1}A\phi(g),g^{-1}\Fcal^\bullet).
\end{equation}
Here the canonical map
$X_\nu\rightarrow \Dfrak_\nu$
is given by 
\[(A,\Fcal^\bullet)\mapsto (\Ocal_{X_\nu}\otimes_{\Q_p}K_0^d, A(\id\otimes\phi),\Fcal^\bullet).\]
\subsection{The weakly admissible locus}
Fix a cocharacter $\nu$ with reflex field $E$ as in the previous section.
If $X\in{\Ad}_E^{\lft}$ and $x\in X$, then our definitions imply that, given $(D,\Phi,\Fcal^\bullet)\in\Dfrak_\nu(X)$, we have 
\begin{align*}
(D\otimes k(x),\Phi\otimes\id,\Fcal^\bullet\otimes k(x)) & \in\Fil\Isoc(k)_{k(x)}^K
\end{align*}
One of the main results of \cite{Hellmann} is concerned with the structure of the weakly admissible locus in the stacks $\Dfrak_\nu$ defined above.
\begin{theo}\label{waopen}
Let $\nu$ be a cocharacter as in $(\ref{cocharnu})$ and $X$ be an adic space locally of finite type over the reflex field of $\nu$.
If $(D,\Phi,\Fcal^\bullet)\in\Dfrak_\nu(X)$, then the weakly admissible locus
\begin{align*}
X^{\rm wa}&=\{x\in X\mid (D\otimes k(x),\Phi\otimes\id,\Fcal^\bullet\otimes k(x))\ \text{is weakly admissible}\}
\end{align*}
is an open subset. Especially it has a canonical structure of an adic space.
\end{theo}
\begin{proof}
This is \cite[Theorem 4.1]{Hellmann}.
\end{proof}
We can define a substack $\Dfrak_\nu^{\rm wa}\subset \Dfrak_\nu$ consisting of the weakly admissible filtered isocrystals. More precisely, for an adic space $X$ the groupoid $\Dfrak_\nu^{\rm wa}(X)$ consists of those triples $(D,\Phi,\Fcal^\bullet)$ such that $(D\otimes k(x),\Phi\otimes\id,\Fcal^\bullet\otimes k(x))$ is weakly admissible for all $x\in X$. Thanks to Corollary $\ref{waextstab}$ it is clear that this is again an fpqc-stack. The following result is now an obvious consequence of Theorem $\ref{waopen}$.
\begin{cor}
The stack $\Dfrak_\nu^{\rm wa}$ on the category of adic spaces locally of finite type over the reflex field of $\nu$ is an open substack of $\Dfrak_\nu$.
\end{cor}

\section{The fibers over the adjoint quotient}

We now come to the main results of this paper. We want to link the weakly admissible locus in 
\[(\Res_{K_0/\Q_p}\GL_d\times\Gr_{K,\nu})^{\ad}\]
as considered in the previous section to the adjoint quotient of the group $\GL_d$. This relation was studied by Breuil and Schneider in \cite{BreuilSchneider}. 
In this section we show that the fibers over the adjoint quotient are (base changes of) analytifications of schemes over $\Q_p$ and hence the period stacks considered here have a much more algebraic nature than the period spaces considered by Rapoport and Zink in \cite{RapoZink}. In the next section we determine the image of the weakly admissible locus in the adjoint quotient and identify it with a closed Newton-stratum in the sense of Kottwitz \cite{Kottwitz}.

First we need to recall some notations and facts about the adjoint quotient from \cite{Kottwitz}.
We write $\GL_d=\GL (V)$ for the general linear group over $\Q_p$, where $V=\Q_p^d$, and $B\subset \GL_d$ for the Borel subgroup of upper triangular matrices.
Further we denote by $A\subset B$ the diagonal torus and identify $X_\ast(A)$ and $X^\ast(A)$ with $\Z^d$.
Let $\Delta=\{\alpha_1,\dots,\alpha_{d-1}\}$ be the simple roots defined by $B$, i.e. $\langle \alpha_i,\nu\rangle=\nu_i-\nu_{i+1}$ for all $\nu\in X_\ast(A)$.
We also choose lifts 
\[\omega_i=(1^{(i)},0^{(d-i)})\in\Z^d=X^\ast (A)\]
of the dual basis $\varpi_1,\dots\varpi_{d-1}\in X^\ast({\rm GL}_d)$ of the coroots. 
Finally $W=S_d$ denotes the Weyl group of $(\GL_d,A)$. 
There is a map 
\[c:A\longrightarrow \Abb^{d-1}\times\Gbb_m\]
which maps an element of $A$ to the coefficients of its characteristic polynomial. This morphism identifies $A/W$ with $\Abb^{d-1}\times\Gbb_m$.

Now we will define a map
\begin{equation}\label{maptoAmodW}
G=\Res_{K_0/\Q_p}(\GL_d)_{K_0}\longrightarrow A/W
\end{equation}
that is invariant under $\phi$-conjugation on the left side.
Recall that we have identifications $\GL_d=\GL(V)$ and $(\GL_d)_{K_0}=\GL(V_0)$, where $V_0=V\otimes_{\Q_p}K_0$.\\
For an $\Q_p$-algebra $R$ and $g\in G(R)$ we have the $R\otimes_{\Q_p} K_0$-linear automorphism $\Phi_g^f=(g(\id\otimes\phi))^f$ of $R\otimes_{\Q_p} V_0$.
Its characteristic polynomial is an element of $(R\otimes_{\Q_p} K_0)[T]$. Now this polynomial is invariant under $\id\otimes\phi$ and hence it already lies in $R[T]$.
We define the morphism in $(\ref{maptoAmodW})$ by mapping $g\in G(R)$ to the coefficients of this polynomial. 
It is easy to check that this morphism is invariant under $\phi$-conjugation on $G$ and hence we get morphisms
\begin{equation}\label{DnunachAmodW}
\begin{aligned}
\begin{xy}\xymatrix{
 G\times\Gr_{K,\nu}\ar[d]\ar[r]^{\widetilde{\alpha}} & (A/W)_E\ar[d]^=\\
 \Dcal_\nu\ar[r]^\alpha &  (A/W)_E,
}\end{xy}
\end{aligned}
\end{equation}
where $\Dcal_\nu$ is the stack-quotient
\[\Dcal_\nu= (G_E\times\Gr_{K,\nu})/ G\]
on the category of $E$-schemes, where the action of $G$ on $G_E\times\Gr_{K,\nu}$ is the same as in $(\ref{phiconjaction})$. Here $\nu$ is a cocharacter as in $(\ref{cocharnu})$ and $E$ is the reflex field of $\nu$.
We also write $\alpha$ and $\widetilde{\alpha}$ for the analytification of these morphisms.
\begin{theo}\label{thmscheme}
Let $x\in (A/W)^{\ad}$ and $\nu$ be a cocharacter as in $(\ref{cocharnu})$.  
Then there exists a quasi-projective scheme $X$ over some finite extension $F$ of $\Q_p$ inside $k(x)$ such that the weakly admissible locus in the fiber over $x$ is given by
\[\widetilde{\alpha}^{-1}(x)^{\rm wa}=X^{\ad}\otimes_F k(x),\]
\end{theo}
We first prove the theorem for rigid analytic points of $(A/W)^{\ad}$. This is the following proposition.
\begin{prop}\label{maintheoforrigpt}
Let $x\in (A/W)^{\ad}$ be a rigid analytic point, then there exists a quasi-projective scheme over $k(x)$ such that
\[\widetilde{\alpha}^{-1}(x)^{\rm wa}=X^{\rm ad}.\]
\end{prop}
\begin{proof}
The proof will be similar to the proof of \cite[Theorem 4.1]{Hellmann}.

Let $x=(c_1,\dots,c_d)\in k(x)^{d-1}\times k(x)^\times$ and let $v_x$ denote that additive valuation on $k(x)$.
First note that $\widetilde{\alpha}^{-1}(x)=\emptyset$ unless 
\[\tfrac{1}{f}v_x(c_d)=\sum_{j\in\Z} j\,\dim\ \gr_j\Fcal^\bullet,\]
where $\Fcal^\bullet$ is the universal filtration on $\Gr_{K,\nu}$.
In the following we will assume that this condition is satisfied.
For $i\in\{0,\dots,d\}$, consider the following functor on the category of $\Q_p$-schemes,
\begin{equation*}
 S\longmapsto
\left\{
{\begin{array}{*{20}c}
 \Ecal\subset \Ocal_S\otimes_{\Q_p}V_0\ \text{locally free}\ \Ocal_S\otimes_{\Q_p}K_0\text{-submodule} \\
\text{of rank}\ i\ \text{that is locally on}\ S\ \text{a direct summand}
\end{array}}
\right\}.
\end{equation*}
This functor is representable by a projective $\Q_p$-scheme $\Gr_{K_0,i}$.\\
We let $G=\Res_{K_0/\Q_p}\GL_d$ act on $\Gr_{K_0,i}$ in the following way: for a $\Q_p$-scheme $S$, let $A\in G(S)$ and $\Ecal\in\Gr_{K_0,i}(S)$. We get a linear endomorphism $A$ of $\Ocal_S\otimes_{\Q_p}V_0$ and define the action of $A$ on $\Ecal$ by
\[
 A\cdot\Ecal=A((\id\otimes\phi)(\Ecal)).
\]
Write
\[
a:G\times\Gr_{K_0,i}\longrightarrow \Gr_{K_0,i}
\]
for this action and consider the subscheme $Z_i\subset G\times\Gr_{K_0,i}$ defined by the following fiber product:
\begin{align*}
 \begin{xy}
  \xymatrix{
Z_i\ar[d]\ar[rr] && G\times\Gr_{K_0,i}\ar[d]^{a\times\id} \\
\Gr_{K_0,i}\ar[rr]^{\hspace{-5mm}\Delta} && \Gr_{K_0,i}\times\Gr_{K_0,i}.
}
 \end{xy}
\end{align*}
An $S$-valued point $x$ of the scheme $Z_i$ is a pair $(g_x,U_x)$, where $g_x\in G(S)$ is a linear automorphism of $\Ocal_S\otimes_{\Q_p}V_0$ and $U_x$ is an $\Ocal_S\otimes_{\Q_p}K_0$-submodule of rank $i$ stable under $\Phi_x=g_x(\id\otimes\phi)$. The scheme $Z_i$ is projective over $G$ via the first projection
\[\pr_i:Z_i\longrightarrow G.\]
Further we denote by $f_i\in\Gamma(Z_i,\Ocal_{Z_i})$ the global section defined by 
\[
 f_i(g_x,U_x)=\dete ((g_x(\id\otimes\phi))^f|_{U_x})
\]
(recall $f=[K_0:\Q_p]$).
We also write $f_i$ for the global section on the associated adic space $Z_i^{\ad}$.\\
We write $\Ecal$ for the pullback of the universal bundle on $Z_i$ to $Z_i\times \Gr_{K,\nu}$ and $\Fcal^\bullet$ for the pullback of the universal filtration on $\Gr_{K,\nu}$. Then the fiber product 
\[\Gcal^\bullet=(\Ecal\otimes_{K_0}K)\cap \Fcal^\bullet\]
is a filtration of $\Ecal\otimes_{K_0}K$ by coherent sheaves. By the semi-continuity theorem the function
\[h_i: x\longmapsto \sum_{i\in\Z}i\ \tfrac{1}{ef}\dim_{\kappa(x)}\gr_i\ \Gcal^\bullet\]
is upper semi-continuous on $Z_i\times \Gr_{K,\nu}$ and hence so is 
\[h_i^{\ad}: x\longmapsto \sum_{i\in\Z}i\ \tfrac{1}{ef}\dim_{k(x)}\gr_i\ (\Gcal^\bullet)^{\ad}.\]
For $m\in\Z$ we write $Y_{i,m}\subset Z_i\times \Gr_{K,\nu}$ (resp. $Y_{i,m}^{\ad}\subset Z_i^{\ad}\times\Gr_{K,\nu}^{\ad}$) for the closed subscheme (resp. the closed adic subspace) 
\begin{align*}
Y_{i,m}&= \{y\in Z_i\times \Gr_{K,\nu}\mid h_i(y)\geq m\},\\
Y_{i,m}^{\ad} &= \{y\in Z_i^{\ad}\times\Gr_{K,\nu}^{\ad}\mid h_i^{\ad}(y)\geq m\}.
\end{align*} 
Then the definitions imply that
\begin{align*}
\pr_{i,m}:Y_{i,m}&\longrightarrow G\times\Gr_\nu\\
\pr_{i,m}:Y_{i,m}^{\ad}&\longrightarrow (G\times\Gr_\nu)^{\ad}
\end{align*}
are proper morphism. Now
\[S_{i,m}=\{y=(g_y,U_y,\Fcal^\bullet_y)\in Y_{i,m}\times_{(G\times\Gr_{K,\nu})}\widetilde{\alpha}^{-1}(x)\mid |f_i(g_y,U_y)|<p^{-f^2m}\}\]
is a union of connected components of $Y_{i,m}$: Let $\lambda_1,\dots,\lambda_d$ denote the zeros of the polynomial
\[X^d+c_1X^{d-1}+\dots+c_{d-1}X+c_d.\]
Then every possible value of the $f_i$ is a product of some of the $\lambda_i$ and hence $f_i$ can take only finitely many values.

We conclude that the subset $\bigcup_{i,m} \pr_{i,m}(S_{i,m})$ is closed and claim that 
\[\widetilde{\alpha}^{-1}(x)^{\rm wa}=\big(\widetilde{\alpha}^{-1}(x)\backslash \bigcup_{i,m} \pr_{i,m}(S_{i,m})\big) ^{\ad}.\]
Indeed, let $z=(g_z,\Fcal^\bullet_z)\in\widetilde{\alpha}^{-1}(x)\subset G^{\rm ad}\times\Gr_{K,\nu}^{\rm ad}$. Then the object 
\[(k(z)\otimes V_0,g_z(\id\otimes\phi),\Fcal_z^\bullet)\]
is not weakly admissible if and only if there exists a $g_z(\id\otimes\phi)$-stable subspace $U_z\subset k(z)\otimes V_0$ of some rank, violating the weak admissibility condition.
This means $z\in\bigcup_{i,m} \pr_{i,m}(S_{i,m})^{\ad}$. Here we implicitly use that fact that weak admissibility is stable under extension of scalars (see Corollary $\ref{waextstab}$).\\ 
\end{proof}
\begin{proof}[Proof of Theorem $\ref{thmscheme}$]
As in the proof of the preceding Proposition we write $x=(c_1,\dots,c_d)\in k(x)^{d-1}\times k(x)^\times$ and denote by $\lambda_1,\dots,\lambda_d\in\Gamma_x\otimes\Q$ the valuations of the zeros of the polynomial 
\[X^d+c_1X^{d-1}+\dots+c_{d-1}X+c_d\]
in some algebraic extension of $k(x)$. Then there exists a rigid analytic point $y\in(A/W)^{\rm ad}(\bar\Q_p)$ such that if we write $\lambda'_1,\dots,\lambda'_d$ 
for the valuations of the associated polynomial we have
\begin{align*}
\tfrac{n-1}{ef} < \prod_{i\in I}\lambda_i <\tfrac{n}{ef} & \Leftrightarrow \tfrac{n-1}{ef}<\prod_{i\in I}\lambda'_i<\tfrac{n}{ef} \\
\tfrac{n}{ef}=\prod_{i\in I}\lambda_i & \Leftrightarrow \tfrac{n}{ef}=\prod_{i\in I}\lambda'_i.
\end{align*}
for all $n\in\Z$ and all $I\subset \{1,\dots,d\}$. This follows from the fact that the set of valuations of finite extensions of $\Q_p$ is dense in $\mathbb{R}$.
Let $L$ denote the composite of $k(x)$ and $k(y)$ and write $F$ for their intersection inside $L$. Then by construction
\begin{align*}
\widetilde{\alpha}^{-1}(x)^{\rm wa}\otimes_{k(x)}L & \subset \Gr_{K,\nu}\otimes_EL\ \text{and}\\
\widetilde{\alpha}^{-1}(y)^{\rm wa}\otimes_{k(y)}L & \subset \Gr_{K,\nu}\otimes_EL
\end{align*}
are defined by exactly the same condition and both are quasi-projective schemes (compare the proof of the proposition above).
It follows that $\widetilde{\alpha}^{-1}(x)^{\rm wa}$ is defined over $F$.
\end{proof}
\begin{rem}
In view of the period domains considered in \cite{RapoZink} it can be surprising that this weakly admissible locus is indeed the adification of a scheme, not just an analytic space. The main reason is the following: In \cite{RapoZink} the isocrystal is fixed and the counter examples one has to exclude for the weak admissibility condition are parametrized by the $\Q_p$-valued points of an algebraic variety. In our setting the isocrystal is not fixed and the Frobenius $\Phi$ may vary. Hence the set of counter examples is the algebraic variety itself rather than its $\Q_p$-valued points.
\end{rem}
\begin{cor}
Let $x\in (A/W)^{\rm ad}$ and consider the $2$-fiber product
\[\begin{xy}\xymatrix{
\alpha^{-1}(x)^{\rm wa} \ar[r]\ar[d] & \Dfrak_\nu^{\rm wa}\ar[d]\\
x \ar[r] & (A/W)^{\rm ad}.
}
\end{xy}\]
Then there exists a finite extension $F$ of $\Q_p$ inside $k(x)$ and an Artin stack in schemes $\mathfrak{A}$ over $F$ such that
\[\alpha^{-1}(x)^{\rm wa}\cong\mathfrak{A}^{\ad}\otimes_Fk(x).\]
\end{cor}
\begin{proof}
This is an immediate consequence of Theorem $\ref{thmscheme}$
\end{proof}
We end the discussion of the fibers over the adjoint quotient  by discussing three examples.
\begin{expl}\label{example1}
Let $K=\Q_p$ and $d=3$. We take $\Phi={\rm diag}(1,p,p^2)$ and fix the type of the filtration $\Fcal^\bullet$ such that
\[\dim \Fcal^i=\begin{cases} 3& i\leq 0 \\ 2 & i=1 \\ 1 & i=2 \\ 0 & i\geq 3. \end{cases}\]
We write $G=\GL_3$ and $B\subset G$ for the Borel subgroup of upper triangular matrices. Further $X=G/B$ is the full flag variety, and we are interested in the weakly admissible locus in $X$. One easily checks that
\[X^{\rm wa}=\{\Fcal^\bullet\in X\mid \Fcal^1\cap V_1=0,\ \text{and}\ \Fcal^2\not\subset V_{12}\},\]
where $0\subset V_1\subset V_{12}\subset \Q_p^3$ is the standard flag fixed by the Borel $B$. The subset $X^{\rm wa}$ is obviously stable under $B$ and, in fact,
\[X^{\rm wa}=Bw_0B/B,\]
where $w_0$ is the longest Weyl group element. 
\end{expl}
\begin{expl}\label{example2}
We use the same notations as in the example above, but this time $\Phi={\rm diag}(1,1,p^3)$. Then
\[X^{\rm wa}=\{\Fcal^\bullet\in X\mid \Fcal^1\cap V_{12}=0\}.\]
As $\dim V_{12}=\dim \Fcal^1=2$ it follows that $X^{\rm wa}=\emptyset$.
\end{expl}
\begin{expl}\label{example3}
In this example let $d=2$ and $\Phi={\rm diag}(1,p)$. Let $K$ be a ramified extension of $\Q_p$ of degree $e$ and consider flags of the type $(1,\dots,1)$, i.e. the cocharacter is defined over $\Q_p$, the only non-trivial filtration step is $\Fcal^1=(\Fcal^1_i)_{i=1\dots,e}$ and the base change of the flag variety $X=\Gr_{K,\nu}$ to $K$ is 
\[X_K=\Pbb_K^1\times\dots\times\Pbb_K^1.\]
The weakly admissible locus in $X_K$ is given by
\[X_K^{\rm wa}=\{\Fcal^\bullet=(\Fcal^\bullet)_i\mid \Fcal^1_i\neq \infty\ \text{for all}\ i\in\{1,\dots,e\}\}.\]
Let $G=\Res_{K/\Q_p}\GL_2$ and $B\subset G$ the Weil-restriction of the Borel subgroup of upper triangular matrices. Again we write $w_0$ for the longest Weyl group element of $G$. Then
\[X^{\rm wa}= Bw_0B/B\subset X=G/B.\]
\end{expl}

\section{Newton strata and weak admissibility}
The proof of Proposition $\ref{maintheoforrigpt}$ suggests that the weakly admissible locus in the fibers over a point in $A/W$ does only depend on the valuation of the zeros of the characteristic polynomial associated to the points of the adjoint quotient. Hence we want to extend the result that the fibers over the adjoint quotient are nice spaces to the pre-image of the Newton strata in the adjoint quotient. Here we work in the category of analytic spaces in the sense of Berkovich (see \cite{Berko}), as it is not obvious how to generalize the notion of Newton strata (as defined in \cite{Kottwitz}) to adic spaces. Though the weakly admissible locus is not a Berkovich spaces in general \cite[Example 4.4]{Hellmann}, we show that it becomes a Berkovich space, if we restrict ourselves to the pre-images of the Newton strata. Further we want to identify the image of the weakly admissible locus in the adjoint quotient with a (closed) Newton-stratum. As usual we will write $\rfH(x)$ for the residue field at a point $x$ in an analytic space and $X^{\an}$ for the analytic space associated to a scheme $X$.
\subsection{Newton strata}
We first need to recall more notations from \cite{Kottwitz}. We write $\afrak=X_\ast(A)\otimes_{\Z}\R$, and $\afrak_{\rm dom}\subset \afrak$ for the subset of dominant elements, i.e the elements $\mu\in\afrak$ such that $\langle\alpha_i,\mu\rangle\geq 0$ for all $i\in\{1,\dots,d-1\}$. For $c=(c_1,\dots,c_d)\in (A/W)^{\an}$ we write
\begin{equation}\label{defdc}
d_c=(-v_c(c_1),\dots, -v_c(c_d))\in\widetilde{\R}^{d-1}\times\R,
\end{equation}
where $v_c$ denotes the (additive) valuation on $\rfH(c)$ and $\widetilde{\R}=\R\cup\{-\infty\}$. Note that there is a sign in $(\ref{defdc})$, as Kottwitz uses a different sign convention.
For $a\in A^{\an}$ define $\nu_a\in\afrak$ by requiring
\begin{equation*}\label{nua}
\langle\lambda,\nu_a\rangle=-v_a(\lambda(a))
\end{equation*}
for all $\lambda\in X^\ast(A)$, where we write $v_a$ for the (additive) valuation on $\rfH(a)$. By \cite[Proposition 1.4.1]{Kottwitz} there is a continuous map $r:\afrak\rightarrow \afrak_{\rm dom}$ mapping $x\in\afrak$ to the dominant element with the smallest distance to $x$, and this map extends in a continuous way to $\widetilde{\R}^{d-1}\times\R$. Here $\afrak\subset \widetilde{\R}^{d-1}\times\R$ via the chosen identification $X_\ast(A)=\Z^d$. Then we find that $r(d_{c(a)})$ is the unique dominant element in the $W$-orbit of $\nu_a$.
This follows from \cite[Theorem 1.5.1]{Kottwitz} for all $a\in A(\bar\Q_p)$ and, for an arbitrary point, from the fact that $A(\bar\Q_p)$ is dense in $A^{\an}$ and the continuity of the construction.
\begin{defn}\label{defNstrata}
For $\mu\in\afrak_{\rm dom}$ we define 
\begin{align*}
 (A/W)_\mu&=\{c\in A/W\mid r(d_c)=\mu\}\\
 (A/W)_{\leq \mu}&=\{c\in A/W\mid r(d_c)\leq \mu\}.
\end{align*}
Here $"\leq"$ is the usual dominance order on dominant coweights.
We will call the first of these subsets the \emph{Newton stratum} defined by $\mu$ and the second the \emph{closed Newton stratum} defined by $\mu$.
\end{defn}
We need another description of these sets to identify them as analytic subspaces of the adjoint quotient.
\begin{prop} Let $\mu\in\afrak_{\rm dom}$ and $I_{\mu}=\{i\in\{1,\dots,d\}\mid \langle\alpha_i,\mu\rangle=0\}$. Then
\begin{align*}
(A/W)_\mu=\left\{c=(c_1\dots,c_d)\in (\Abb^{d-1}\times \Gbb_m)^{\an} \left | 
\begin{array}{*{20}c}
v_c(c_i)\geq -\langle\omega_i,\mu\rangle\ ,i\in I_\mu\\
v_c(c_i)= -\langle\omega_i,\mu\rangle\ ,i\notin I_\mu
 \end{array}\right.\right\}\\
(A/W)_{\leq \mu}=\left\{c=(c_1\dots,c_d)\in (\Abb^{d-1}\times \Gbb_m)^{\an} \left | 
\begin{array}{*{20}c}
v_c(c_i)\geq -\langle\omega_i,\mu\rangle\ ,i\neq d\\
v_c(c_d)= -\langle\omega_d,\mu\rangle.  \end{array}\right.\right\} 
\end{align*}
\end{prop}
\begin{proof}
For all points in $(A/W)^{\an}(\bar\Q_p)$ this follows from \cite[Theorem 1.5.2]{Kottwitz}. Again the proposition follows from continuity, and the fact that the points in $(A/W)^{\an}(\bar\Q_p)$ are dense in $(A/W)^{\an}$.
\end{proof}
The category of (strict) analytic spaces is a full subcategory of the category of adic spaces locally of finite type. Hence we can restrict the stacks $\Dfrak_\nu$ and $\Dfrak_\nu^{\rm wa}$ to the category of analytic spaces. We write again $\Dfrak_\nu$ and $\Dfrak_\nu^{\rm wa}$ for these restrictions. Further we write $\widetilde{\alpha}^{\an}$ (resp. $\alpha^{\an}$) of the analytifications of the  morphisms defined in $(\ref{DnunachAmodW})$.
\begin{theo}
Let $\nu$ be a cocharacter as in $(\ref{cocharnu})$ and $\mu\in\afrak_{\rm dom}$. Then the weakly admissible locus in the inverse image $(\widetilde{\alpha}^{\an})^{-1}((A/W)^{\an}_\mu)$ is an analytic space.
\end{theo}
\begin{proof}
The proof is almost identical with the proof of Theorem $\ref{thmscheme}$.
Here the functions $f_i$ are not locally constant, but their valuations (or absolute values) are.
\end{proof}
\subsection{The image of the weakly admissible locus}
In this section we determine the image of the weakly admissible locus under the map defined in $(\ref{maptoAmodW})$.
In the case of a regular cocharacter $\nu$ it was shown by Breuil and Schneider that the set of points $a\in A$ such that there exists a weakly admissible filtered $\phi$-module $(D,\Phi,\Fcal^\bullet)$ with $(\Phi^f)^{\rm ss}=a$ is an affinoid domain, see \cite[Proposition 3.2]{BreuilSchneider}.
Here we extend this result to the general case and give a description of this image purely in terms of the adjoint quotient $A/W$. The difference with the description in \cite{BreuilSchneider} is that we do not need to fix an order of the generalized eigenvalues corresponding to the order of their valuations.
We fix a coweight $\nu$ as in $(\ref{cocharnu})$. This coweight determines the jumps of the filtration $\Fcal^\bullet$ on $\Gr_{K,\nu}$. 
After passing to $\bar\Q_p$ the filtration is given by
\[\Fcal^\bullet=\prod_\psi\Fcal_\psi^\bullet\]
where the product runs over all embeddings $\psi:K\hookrightarrow \bar\Q_p$.\\
We write $\{x_{\psi,1}>x_{\psi,2}>\dots >x_{\psi,r}\}$ for the jumps of the filtration $\Fcal^\bullet_\psi$, i.e. 
\[\gr_i\Fcal_\psi^\bullet\neq 0\Leftrightarrow i\in\{x_{\psi,1},\dots,x_{\psi,r}\}.\]
Further denote by $n_{\psi,i}$ the rank of $\Fcal_\psi^{x_{\psi,i}}$. For $i\in\{0,\dots,d\}$ define 
\begin{equation}\label{defli}
l_i=\sum_{\psi}\tfrac{1}{ef}\left(\sum _{j=1}^{r-1}(x_{\psi,j}-x_{\psi,j+1})\max(0,n_{\psi,j}+i-d)+x_{\psi,r}d\right).
\end{equation}
\begin{lem}\label{technischeslem} 
Let $s_1,s_2\in\R$ and $j_1\geq j_2\in\mathbb{N}$. Let $i\in\{1,\dots,d-j_1\}$ such that
$s_1\geq l_i$ and $s_1+j_1s_2\geq l_{i+j_1}$. Then $s_1+j_2s_2\geq l_{i+j_2}$.
\end{lem}
\begin{proof}
We view $f_1(j)=s_1+js_2$ and $f_2(j)=l_{i+j}$ as functions of $j\in[0,j_1]$. Then $f_1$ is (affine) linear while $f_2$ is picewise linear and convex. As $f_1(0)\geq f_2(0)$ and $f_1(j_1)\geq f_2(j_1)$ by assumption, the claim follows. 
\end{proof}
\begin{defn}\label{defmunu}
For a cocharacter $\nu$ and $i\in\{1,\dots,d\}$ define $l_i$ as in $(\ref{defli})$. Define a rational dominant coweight $\mu(\nu)\in\afrak_{\rm dom}$ by requiring that
\[\tfrac{1}{f}\langle\omega_i,\mu(\nu)\rangle=-l_i\ \text{for all}\ i\in\{1,\dots,d\}.\]
\end{defn}
The following result generalizes \cite[Proposition 3.2]{BreuilSchneider}.
\begin{theo}
Let $\nu$ be a cocharacter as in $(\ref{cocharnu})$ and define $\mu(\nu)$ as in Definition $\ref{defmunu}$.
Then $(\widetilde{\alpha}^{\an})^{-1}(x)^{\rm wa}\neq\emptyset$ if and only if $x\in(A/W)^{\an}_{\leq \mu(\nu)}$.
\end{theo}
\begin{proof}
Let $c=(c_1,\dots,c_d)\in(A/W)^{\an}_{\leq \mu(\nu)}$ and denote by $\lambda_1,\dots,\lambda_t$ the roots of 
\[X^d+c_1X^{d-1}+\dots+c_{d-1}X+c_d\]
with multiplicities $m_i$ in some finite extension $L$ of $\rfH(c)$ containing $K_0$. Let $D=L\otimes_{\Q_p}V_0\cong \prod_{i=1}^d L^d$ and
\[g=(\id,\dots,\id,A)\in \prod_{i=1}^d\GL_d(L)\cong G(L),\]
where $A$ is a matrix consisting of $t$ Jordan blocks of size $m_i$ with diagonal entries $\lambda_i$. 
Now the pair $(D,\Phi)=(D,g(\id\otimes\phi))\in\Isoc(k)_L$ has the property that there are only finitely many $\Phi$-stable subobjects $D'\subset D$.
If $D'\subset D$ is a rank $i$ subobject then
\[t_N(D'):=\tfrac{1}{f^2}v_p(\dete \Phi^f|_{D'})=\tfrac{1}{f}\sum_{j=1}^t m'_jv_p(\lambda_j)\]
for some multiplicities $m'_j$, where we write $v_p$ for the additive valuation.  Write $a=(\lambda_1^{(m_1)},\dots,\lambda_t^{(m_t)})\in A^{\an}$, then
\begin{align*}
t_N(D')&=\tfrac{1}{f}v_p(w\omega_i(a))\hspace{-2cm}&& =-\tfrac{1}{f}\langle w\omega_i,\nu_a\rangle \\
&=-\tfrac{1}{f}\langle w'\omega_i,r(d_{c(a)})\rangle \hspace{-2cm}&& \geq -\tfrac{1}{f}\langle \omega_i,r(d_{c(a)})\rangle \\
&\geq -\tfrac{1}{f}\langle \omega_i,\mu(\nu)\rangle \hspace{-2cm}&&=l_i.
\end{align*} 
for some $w,w'\in W$.
Now for all $\Phi$-stable $D'\subset D$ consider the open subset 
\[U_{D'}\subset \Gr_{K,\nu}\otimes_{\Q_p}K_0\]
of all filtrations $\Fcal^\bullet$ such that $\dim (\Fcal^{x_j}_\psi\cap D'_K)=\max(0,n_{\psi,j}+i-d)$ for all embeddings $\psi$. This is open as the right hand side is the minimal possible dimension of such an intersection. Since $\Gr_{K,\nu}$ is geometrically irreducible we find that the intersection $\bigcap_{D'\subset D} U_{D'}$ is non-empty and hence there exists an $F$-valued point $\Fcal^\bullet$ in this intersection, where $F$ is some extension of $L$. Now we have $(D\otimes_{K_0}F,\Phi\otimes\id,\Fcal^\bullet)\in\Fil\Isoc(k)_F^K$ and this object is weakly admissible since for all $\Phi$-stable $D'\subset D$ we have
\[\deg(D')=l_i-t_N(D')\leq 0\]
where $i$ is the rank of the subobject $D'$ (and here we write the degree additively). Further, by the definition of $g$, we find that $g$ maps to $c$ under the map $\widetilde{\alpha}$.  

Conversely assume that $c\in(A/W)^{\an}$ such that $\emptyset\neq\widetilde{\alpha}^{-1}(c)$. Let $(D,\Phi,\Fcal^\bullet)$ be an $F$-valued point of this fiber for some field $F$ containing $K_0$. Then $D$ decomposes into $D_1\times\dots\times D_f$ and we denote by $\mu_1,\dots,\mu_t$ the distinct generalized eigenvalues of $\Phi^f|_{D_1}$ and by $d_i$ their multiplicities. We write $(\lambda_1,\dots,\lambda_d)=(\mu_1^{(d_1)},\dots,\mu_t^{(d_t)})$.
Then $c=c(\lambda_1,\dots\lambda_d)$ and we claim that
\[\tfrac{1}{f}\sum_{j\in I}v_p(\lambda_j)\geq l_i\]
for all $I\subset \{1,\dots,d\}$. This claim clearly implies $c\in (A/W)^{\an}_{\leq \mu(\nu)}$.\\
Let $I\subset \{1,\dots,d\}$ and write $(\lambda_i)_{i\in I}=({\lambda'_1}^{(m_1)},\dots,{\lambda'_t}^{(m_t)})$, where we assume that the $\lambda'_i$ are pairwise distinct. Then $\sum_{j=1}^tm_j=i=\sharp I$. 
There exists a subobject $D'\subset D$ and $m'_j\geq m_j$ such that
\[(\Phi^f|_{D'\cap D_1})^{\rm ss}={\rm diag}({\lambda'_1}^{(m'_1)},\dots{\lambda'_t}^{(m'_t)})\]
and hence 
\[\sum_{j\geq 1} m'_j\tfrac{1}{f}v_p(\lambda'_j)\geq \sum_{j\geq 1} j\tfrac{1}{ef}\dim \gr_i D'_K\geq l_{i_1}.\]
where $i_1$ denotes the rank of $D'$.  
Apply Lemma $\ref{technischeslem}$ with $j_1=m'_1$, $j_2=m_1$ and 
\begin{align*}
& s_1=\sum_{j\geq 2}m'_j\tfrac{1}{f}v_p(\lambda'_j)\\
& s_2=\tfrac{1}{f}v_p(\lambda'_1),
\end{align*}
(the condition on the indices is obviously satisfied). This yields
\[m_1\tfrac{1}{f}v_p(\lambda'_1)+\sum_{j\geq 2}m'_j\tfrac{1}{f}v_p(\lambda'_j)\geq l_{i_1-m'_1+m_1}.\]
As similar argument shows 
\[m_1\tfrac{1}{f}v_p(\lambda'_1)+\sum_{j\geq 3}m'_j\tfrac{1}{f}v_p(\lambda'_j)\geq l_{i'_1-m'_1+m_1},\]
where $i'_1$ is the rank of the direct sum of all generalized eigenspaces of $\Phi^f$ on $D'$ with eigenvalues different from $\lambda_2$.  Applying Lemma $\ref{technischeslem}$ again, we conclude
\[m_1\tfrac{1}{f}v_p(\lambda'_1)+m_2\tfrac{1}{f}v_p(\lambda'_2)+\sum_{j\geq 3}m'_j\tfrac{1}{f}v_p(\lambda'_j)\geq l_{i_1-m'_1+m_1-m'_2+m_2}.\]
Repeating this argument several times we finally end up with the claim.
\end{proof}
We end by giving two examples of closed Newton strata in the adjoint quotient.
\begin{expl}
Let $K=\Q_p$ and $d=3$. We fix the cocharacter $\nu$ as in Example $\ref{example1}$ and Example $\ref{example2}$, i.e.
\[\dim \Fcal^i=\begin{cases} 3 & i\leq 0 \\ 2& i=1 \\ 1&i=2 \\ 0&i\geq 3. \end{cases}\]
One easily checks that $l_1=0$, $l_2=1$ and $l_3=3$, i.e.
\[\mu(\nu):t\longmapsto {\rm diag}(1,t^{-1}, t^{-2}).\]
The image of the weakly admissible locus in the adjoint quotient is given by
\[(A/W)^{\an}_{\leq\mu(\nu)}=\left \{c=(c_1,c_2,c_3)\in \Abb^2\times\Gbb_m\left| \begin{array}{*{20}c} v_c(c_1)\geq 0, \\ v_c(c_2)\geq 1, \\ v_c(c_3)=3. \end{array}\right.\right\}\]
If $a=(a_1,a_2,a_3)\in A$ with $v_a(a_1)\leq v_a(a_2)\leq v_a(a_3)$, then \cite[Proposition 3.2]{BreuilSchneider} says that there exists a weakly admissible filtered $\phi$-module $(D,\Phi,\Fcal^\bullet)$ with filtration of type $\nu$ such that $\Phi^{\rm ss}=a$ if and only if 
\[\begin{cases} 0\leq v_a(a_1) \\ 1\leq v_a(a_1)+v_a(a_2) \\ 3=v_a(a_1)+v_a(a_2)+v_a(a_3).\end{cases}\]
This is clearly equivalent to our condition in the adjoint quotient. This result also explains Example $\ref{example1}$ and Example $\ref{example2}$.
\end{expl}
\begin{expl}
Again we let $K=\Q_p$ and $d=3$. Fix a cocharacter $\nu$ such that
\[\dim \Fcal^i=\begin{cases}3&i\leq 3 \\ 2& i=1 \\ 0 & i\geq 2.\end{cases}\]
One easily checks that $l_1=0$, $l_2=1$ and $l_3=2$, and the image of the weakly admissible locus is
\[(A/W)^{\an}_{\leq\mu(\nu)}=\left \{c=(c_1,c_2,c_3)\in \Abb^2\times\Gbb_m\left| \begin{array}{*{20}c} v_c(c_1)\geq 0, \\ v_c(c_2)\geq 1, \\ v_c(c_3)=2. \end{array}\right.\right\}\]
\end{expl}

\medskip
\noindent
\address{Mathematisches Institut der Universit\"at Bonn\\ Endenicher Allee 60, 53115 Bonn, Germany}\\
\email{hellmann@math.uni-bonn.de}

\end{document}